\documentclass{amsart}
\usepackage{amsfonts}
\usepackage{amsmath,amssymb}
\usepackage{amsthm}
\usepackage{amscd}
\usepackage{graphics}
\usepackage{graphicx}

\theoremstyle{remark}{
\newtheorem{Def}{{\rm Definition}}[section]
\newtheorem{Ex}{{\rm Example}}[section]
\newtheorem{Rem}{{\rm Remark}}[section]

\newtheorem{MainProb}{Main Problem}
}
\theoremstyle{plain}{

\newtheorem{Prop}{Proposition}[section]
\newtheorem{Thm}{Theorem}[section]
\newtheorem{MainThm}{Main Theorem}

}

\begin{document}

\title[On $2$-dimensional Reeb spaces of simple fold maps]{
Branched surfaces homeomorphic to Reeb spaces of simple fold maps}
\author{Naoki Kitazawa}
\keywords{Branched surfaces. Reeb spaces. Polyhedra. Curves and surfaces. Immersions and embeddings. Morse functions and fold maps. \\
\indent {\it \textup{2020} Mathematics Subject Classification}: Primary~57Q05. Secondary~57Q35, 57R45.}
\address{Institute of Mathematics for Industry, Kyushu University, 744 Motooka, Nishi-ku Fukuoka 819-0395, Japan \\
 TEL (Office): +81-92-802-4402 \\
 FAX (Office): +81-92-802-4405 \\
}
\email{n-kitazawa@imi.kyushu-u.ac.jp}
\urladdr{https://naokikitazawa.github.io/NaokiKitazawa.html}
\maketitle
\begin{abstract}
Classes of {\it branched surfaces} extend the classes of {\it surfaces} or 2-dimensional manifolds satisfying suitable properties and defined in various manners.
{\it Reeb spaces} of smooth maps of suitable classes into surfaces whose codimensions are negative 
are regarded as branched surfaces. They are the spaces of all connected components of preimages and natural quotient spaces of the manifolds of the domains. They are defined for general smooth maps and important topological objects in differential topology. They also play important roles in applied or applications of mathematics such as projections in data analysis and visualizations. 

The present paper concerns global topologies of branched surfaces and explicit construction of canonically obtained maps from the branched surfaces into surfaces of the targets via fundamental operations. The class of these induced maps extends the class of smooth immersions of compact surfaces into surfaces with no boundaries. It is also regarded as a variant of the class of so-called generic smooth maps between these surfaces. We study so-called "geography" of such maps as a natural, important and new study and also study global topological properties of the branched surfaces such as embeddability into $3$-dimensional closed and connected manifolds.




\end{abstract}


\maketitle
\section{Introduction.}
\label{sec:1}

{\it Branched surfaces} are defined in various manners suitable for given situations, extending the classes of {\it surfaces} of suitable classes. They are also higher dimensional variants of graphs for example. Branched surfaces are central objects in the present paper and we study global topological properties of them for example.

{\it Reeb spaces} of smooth maps of suitable classes into surfaces whose codimensions are negative 
are regarded as branched surfaces where the suitable definitions are given.

For a continuous map $c:X \rightarrow Y$, the equivalence relation ${\sim}_c$ on $X$ is defined in the following way: $x_1 {\sim}_c x_2$ holds for $x_1,x_2 \in X$ if and only if they are in a same connected component of a preimage. 
\begin{Def}
	Here the quotient space $W_c:=X/{\sim}_c$ of $X$ is the {\it Reeb space} of $c$.
\end{Def}

In addition, $q_f:X \rightarrow W_c$ and $\bar{c}$ denote the quotient map and the map uniquely defined for a natural relation $c=\bar{c} \circ q_c$.

In this paper, we mainly discuss the PL category, or as an equivalent category, the piesewise smooth category.
A {\it polyhedron} means an object in these categories. 
For an object $X$ there, $\dim X$ debotes the dimension of $X$. A morphism in the PL category is called a {\it PL} map generally and one in the piesewise smooth category is called a {\it piesewise smooth} map. Two polyhedra are {\it PL homeomorphic} if there exists a homeomorphism which is a PL map or a piecewise smooth map. We call such a map a {\it PL homeomorphism} and {\it piecewise smooth homeomorphism}, respectively. We can define a {\it PL embedding} and a {\it piesewise smooth embedding} for example. We can define a {\it subpolyhedron} and as a specific case a {\it PL submanifold} for example. For a regular neighborhood
of a subpolyhedron, we can define a {\it regular neighborhood} as a neighborhood which is also a subpolyhedron. It is unique under a natural equivalence relation. For these categories, see also \cite{hudson} and see also \cite{bryant} for example. We do not expect much knowledge on these categories and this is no problem.

The Reeb spaces of smooth functions and smooth maps whose codimensions are negative are fundamental and important topological objects in (differential) topology of manifolds, representing manifolds compactly (\cite{reeb}). 
They are polyhedra whose dimensions are same as those of the targets in considerable cases. Reeb spaces(, especially ones whose dimensions are $1$ or $2$,) are also fundamental and strong tools in applied (or applications of) mathematics, such as projections in data analysis and visualizations for example.

The present paper also concerns Reeb spaces of smooth maps of suitable classes on closed manifolds whose dimensions are greater than $2$ (into surfaces).

Throughout the present paper, ${\mathbb{R}}^k$ denotes the $k$-dimensional Euclidean space for each integer $k \geq 0$: $\mathbb{R}$ denotes the field of all real numbers and ${\mathbb{R}}^1$ and $\mathbb{R}$ are same.
$\mathbb{Z} \subset \mathbb{R}$ denotes the ring of all integers which is also a subring of $\mathbb{R}$. For each point or vector $x \in {\mathbb{R}}^k$, $||x|| \geq 0$ denotes the distance between $x$ and the origin $0$ and the value of the canonically defined norm at $x$ where ${\mathbb{R}}^k$ is a smooth manifold with the natural differentiable structure and the Riemannian manifold endowed with the standard Euclidean metric. $S^k:=\{x \in {\mathbb{R}}^{k+1} \mid ||x||=1.\}$ denotes the $k$-dimensional unit sphere and it is a smooth closed submanifold of ${\mathbb{R}}^{k+1}$ with no boundary for $k \geq 0$.
$D^k:=\{x \in {\mathbb{R}}^{k} \mid ||x|| \leq 1.\}$ denotes the $k$-dimensional unit disk and it is a compact smooth submanifold of ${\mathbb{R}}^{k}$ for $k \geq 1$.

The following class of $2$-dimensional compact polyhedra is also presented in \cite{kitazawa11} for example. 
\cite{ikeda} presents a related classical theory. This class has been important in various scenes of geometry such as ones in topological theory of $3$-dimensional manifolds: related studies are presented shortly later. 

We define a polyhedron in this class as a {\it branched surface} in the present study.
A {\it PL} bundle means a bundle whose fiber is a polyhedron and whose structure group consists of PL homeomorohisms or piecewise smooth homeomorphisms.

\begin{Def}
\label{def:2}
A {\it branched surface} $P$ is a $2$-dimensional compact polyhedron satisfying the following three.
\begin{enumerate}
\item There exists a family $\{C_j\}$ of finitely many copies of circles which are subpolyhedra of $P$ and mutually disjoint.
\item $P-{\sqcup}_{j} C_j$ is a $2$-dimensional manifold with no boundary.
\item For each $C_j$ in the family $\{C_j\}$, some small regular neighborhood $N(C_j)$ has either of the following properties.
\begin{enumerate}
\item $N(C_j)$ is the total space of a trivial PL bundle over $C_j$ whose fiber is PL homeomorphic to $[0,1]$ where $C_j \subset N(C_j) \subset P$ is identified with $C_j \times \{0\} \subset C_j \times [0,1]$.
\item Let $K:=\{(r\cos t, r\sin t) \mid 0 \leq r \leq 1,t=0,\frac{2}{3}\pi,\frac{4}{3}\pi.\} \subset {\mathbb{R}}^2$. $N(C_j)$ is the total space of a PL bundle over $C_j$ whose fiber is PL homeomorphic to $K$ and whose structure group is a trivial group or a group of order $2$ acting in the following way: the action by the non-trivial group fixes $\{(r,0) \mid 0 \leq r \leq 1.\}$ and the non-trivial element maps $(r\cos \frac{2}{3}\pi, r\sin \frac{2}{3}\pi)$ to $(r\cos \frac{4}{3}\pi, r\sin \frac{4}{3}\pi)$. $C_j \subset N(C_j) \subset P$ is identified with $C_j \times \{0\}$.
\end{enumerate}
\end{enumerate}
\end{Def}

A {\it PL manifold} means a polyhedron which is a (topological) manifold and satisfying certain conditions on {\it simplicial structures}, compatible with the polyhedron. It is also well-known that a smooth manifold is canonically regarded as a PL manifold. A topological space homeomorphic to a $1$ or $2$-dimensional polyhedron is always homeomorphic to a unique polyhedron and a topological manifold whose dimension is $1$, $2$ or $3$ is well-known to be homeomorphic to a unique smooth manifold and a unique PL manifold. This is so-called hauptvermutung, discussed in \cite{moise} for example.

 Throughout the present paper, {\it surfaces} mean $2$-dimensional PL manifolds possibly with non-empty boundaries. A compact surface is regarded as a branched surface.

In the present paper, we also concentrate on the class of so-called {\it simple} {\it fold} maps into the plane or surfaces with no boundaries on closed manifolds whose dimensions are greater than $2$ and their Reeb spaces. They are branched surfaces in Definition \ref{def:2}. Here we shortly explain about {\it simple} fold maps, leaving rigorous expositions to the third section. Such maps form a certain class of smooth maps, generalizing the class of so-called {\it Morse} functions. We do not explain about Morse functions precisely and see \cite{milnor,milnor2} for example. 

For a differentiable map $c:X \rightarrow Y$, a {\it singular} point $p \in X$ is defined as a point where the rank of the differential ${dc}_p$ is smaller than $\min\{\dim X,\dim Y\}$ and we call the set of all singular points of $c$ the {\it singular set} of $c$, denoted by $S(c)$. $c(S(c))$ is the {\it singular value set} of $c$.

The singular sets of {\it Morse} functions are discrete subsets.

For such maps into surfaces with no boundaries, the Reeb spaces satisfy Definition \ref{def:2}.

The present paper also concerns the global topology of the Reeb space of a given simple fold map $f$ on a closed manifold whose dimension is greater than $2$ into a surface and global properties of the map $\bar{f}$. The following problems are fundamental questions in our study.

\begin{MainProb}
\label{mprob:1}
Is a given branched surface regarded as a polyhedron homeomorphic to the Reeb space of a simple fold map on a closed manifold whose dimension is greater than $2$ into some surface with no boundary?
\end{MainProb}
This is or can be, in some senses, affirmatively solved under suitable situations. 

First, this is divided into problems on existence of smooth immersions of circles and compact surfaces into surfaces. See the last part of the third section for related articles on the existence theory of such maps, for example.

Moreover, for example, \cite{ishikawakoda} gives construction of simple fold maps from branched surfaces with several associated data giving maps onto the branched surfaces on $3$-dimensional closed manifolds locally regarded as $q_f$ of a simple fold map $f$: this is an application of the theory of so-called {\it shadows}.

In other scenes, for (finite) graphs, affirmative answers are presented in various situations in various articles.
\cite{sharko} is a pioneering work related to the problems, showing that a finite graph satisfying suitable conditions is homeomorphic to the Reeb space of a smooth function on some closed surface via explicit construction of functions. For example, 
\cite{masumotosaeki, michalak, saeki5} are related studies. The author has obtained related results in \cite{kitazawa0.5, kitazawa0.6, kitazawa5} for example.
\begin{MainProb}
\label{mprob:2}
For the class of simple fold maps into surfaces with no boundaries on closed manifolds whose dimensions are greater than $2$, study explicit construction of a simple fold map $f$ and the induced map $\bar{f}$ and so-called ''geography'' of such maps. 
\end{MainProb}

If the Reeb space of a map $f$ of such a class is a compact surface, then $\bar{f}$ is a smooth immersion. As in Proposition \ref{prop:4}, for a smooth immersion of a compact surface into surfaces, we have a simple fold map $f$ such that $\bar{f}$ is the original immersion.
This problem can be a study on classes of maps on branched surfaces into surfaces wider than the class of smooth immersions of compact surfaces into surfaces. 

Our study is also and mainly motivated by a large amount of studies on Morse functions and higher dimensional variants and applications to (differential) topology of manifolds including the studies just before. 
Especially, studies on construction of explicit smooth maps, which are important in the theory and have been difficult, and related works on global algebraic topological or differential topological properties of Reeb spaces by the author, have led us to the present study.    
\cite{kitazawa2, kitazawa3, kitazawa4, kitazawa6, kitazawa8, kitazawa9, kitazawa11} are some of related studies by the author. Studies concentrating more on manifolds admitting these maps are \cite{kitazawa0.1, kitazawa0.2, kitazawa0.4, kitazawa1, kitazawa7, kitazawa10} for example. 
More studies related to the theory will be presented later when we need to present.

\begin{MainProb}
\label{mprob:3}
What can we say about global topological properties of our branched surfaces? For example, can we embed them into given $3$-dimensional closed and connected manifolds?
\end{MainProb}

This is a well-known variant of embeddability of graphs into surfaces. See \cite{ozawa} and see also \cite{matsuzakiozawa,munozozawa}.

The present paper is organized as follows.

In the next section, we introduce Main Theorems with some short expositions.

The third and fourth sections are precise expositions on Main Theorems starting from rigorous expositions on fundamental terminologies and notions.

\section{Main Theorems.}

We explain shortly about some notions to introduce Main Theorems.

A {\it normal} branched surface is a branched surface where the bundles over $C_i$ whose fibers are PL homeomorphic to $K$ in Definition \ref{def:2} are always trivial, defined again in Definition \ref{def:6}, in the next section. An {\it SSNS} fold map is a simple fold map on an $m$-dimensional closed manifold into a surface with no boundary satisfying $m>2$ such that preimages containing no singular points are disjoint unions of copies of a unit sphere, that the Reeb space is normal, and that additional suitable properties are enjoyed. This is also defined rigorously in Definition \ref{def:6}. We also need a map ({\it locally}) {\it born from an SSNS fold map}, defined in Definition \ref{def:7}, in the next section. This is a map locally regarded as the quotient map onto the Reeb space induced from an SSNS fold map.



Main Theorems \ref{mthm:1}, \ref{mthm:2.1} and \ref{mthm:2.2} are on new global topological properties of the Reeb space $W_f$ and construction of the map $\bar{f}$ in Main Problems \ref{mprob:1} and \ref{mprob:2}.

For the following theorem, fundamental notions on elementary algebraic topology are presented in the fourth section before presenting this as Theorem \ref{thm:1} again. We also expect that readers have some knowledge on them. \cite{hatcher} explains about them systematically. 
$B(c)$ denotes the disjoint union ${\sqcup}_j C_j$ for the branched surface $W_c$ in Definition \ref{def:2}. We define this again in the fourth section.
\begin{MainThm}[Theorem \ref{thm:1}]

	\label{mthm:1}
Let $P$ be a branched surface, $N$ a surface with no boundary, and $c:P \rightarrow N$ a {\it map }{\rm (}{\it locally}{\rm )} {\it born from an SSNS fold map}.
Let $\{T_j\}_j$ be a family of finitely many circles which are disjointly embedded in $P-B(c)$ smoothly and regarded as subpolyhedra of $P$. Let the size of the family $\{T_j\}_j$ be $l>0$.
Let ${\bigcup}_{j} c(T_j)$ be the image of the boundary of a compact and connected surface $S_C$ smoothly immersed into $N$.
Then by attaching a surface homeomorphic to $S_C$ along ${\sqcup}_j T_j$ on the boundary by a PL homeomorphism, a piesewise homeomorphism, or a diffeomorphism, we have a new 
{\rm (}normal{\rm )} branched surface $P^{\prime}$ and a map $c^{\prime}:P^{\prime} \rightarrow N$ {\rm (}resp. locally{\rm )} born from an SSNS fold map.  

Furthermore, we have the following two where $A$ is a principal ideal domain having a unique identity element different from the zero element. 
\begin{enumerate}
	\item 
	\label{mthm:1.1}
	Let the value of the homomorphism induced by the inclusion of the {\rm (}suitably oriented{\rm )} circle $T_j$ at the fundamental class is the zero element of $H_1(P;A)$. Assume that $P$ is connected. Assume also that $S_C$ is orientable. Then we have the following information on the homology groups and the cohomology groups.
	\begin{enumerate}
		\item
		\label{mthm:1.1.1}
		$H_1(P^{\prime};A)$ is isomorphic to and identified with the direct sum of $H_1(P;A)$, $A^{l-1}$, and $H_1(S_{C,0};A)$ where $S_{C,0}$ denotes a closed and connected surface obtained by attaching copies of the $2$-dimensional unit disk to $S_{C}$ by a PL homeomorphism, a piesewise homeomorphism, or a diffeomorphism between the boundaries. Furthermore, this cohomology group is free as a module over $A$. 
		\item
		\label{mthm:1.1.2}
		$H_2(P^{\prime};A)$ is isomorphic to and identified with the direct sum of $H_2(P;A)$ and $A$.
		\item
		\label{mthm:1.1.3}
		The cohomology group $H^{1}(P^{\prime};A)$ is isomorphic to and identified with the direct sum of $H^{1}(P;A)$, $H^{1}(S_{C,0};A)$, and $A^{l-1}$ where $A^{l_0}$ denote the free module over $A$ of rank $l_0 \geq 0$. If $H_1(P^{\prime};A)$ is free, then the cohomology group $H^2(P^{\prime};A)$ is isomorphic to the direct sum of $H^2(P;A)$ and $A$.		
	\end{enumerate}
	\item
	\label{mthm:1.2}
	If $A$ consists of elements whose orders are at most $2$, then we do not need to assume that $S_C$ is orientable in the previous situation.
\end{enumerate}
\end{MainThm}

A {\it piesewise smooth homotopy} means a homotopy considered in the piesewise smooth category. 

\begin{MainThm}[Theorem \ref{thm:2} (\ref{thm:2.1})]
\label{mthm:2.1}
Let $P$ be a branched surface, $N$ a surface with no boundary, and $c_0:P \rightarrow N$ a map {\rm (}locally{\rm )} born from an SSNS fold map.
Let $\{D_j\}$ be a family of finitely many copies of the $2$-dimensional unit disk which are disjointly embedded in $P-B(c_0)$ smoothly and regarded as subpolyhedra of $P$.
We also assume that $N$ is connected and that $N-c_0(P)$ is not empty. Then by a suitable piesewise smooth homotopy $F_c$ from $c_0$ to a new map $c:P \rightarrow N$ and apply Theorem \ref{mthm:1} by setting $T_j:=F_c(\partial D_j \times \{1\})$ and $S_C$ as a compact, connected and orientable surface of genus $0$.
Furthermore, if $P$ is connected, then we can also apply additional statements on homology groups and cohomology groups in Main Theorem \ref{mthm:1}.
\end{MainThm}
In the following main theorem, the notion of the {\it Heegaard genus} for a $3$-dimensional closed, connected and orientable manifold appears. 
In short, such a manifold is decomposed into two copies of a manifold diffeomorphic to one represented as a boundary connected sum of finitely many copies of $S^1 \times D^2$ along a closed, connected and orientable surface, which is the boundaries of these $3$-dimensional manifolds. The minimal number of the copies of $S^1 \times D^2$ we need is the {\it Heegaard genus} of the original $3$-dimensional closed, connected and orientable manifold. We can define this notion for the non-orientable case and in this case $S^1 \times D^2$ is replaced by the total space of a (PL) bundle over $S^1$ whose fiber is (PL) homeomorphic to $D^2$ and which is non-orientable as a $3$-dimensional manifold.

For example, the Heegaard genus of a $3$-dimensional sphere is $0$. A manifold represented as a connected sum of $g>0$ copies of $S^1 \times S^2$ is $g$. More generally, one represented as a connected sum of $g$ copies of the total space of a bundle over $S^1$ whose fiber is (PL) homeomorphic to $S^2$ is $g>0$.
See \cite{hempel} for example. We omit rigorous arguments on Heegaard genera of $3$-dimensional closed and connected manifolds in the present paper and see this for these arguments.
\begin{MainThm}[Theorem \ref{thm:2} (\ref{thm:2.2})]
\label{mthm:2.2}
In Main Theorem \ref{mthm:2.1}, we have the following two.
\begin{enumerate}
	\item Assume that we can embed $P$ as a PL or smooth submanifold in a $3$-dimensional closed and connected manifold of Heegaard genus $g$. Then we can construct and have a PL {\rm (}piesewise smooth{\rm )} embedding of $P^{\prime}$ into one whose Heegaard genus is $g+l-1$.
	\item Assume that we can embed $P$ as a PL or smooth submanifold in a $3$-dimensional closed, connected and orientable manifold of Heegaard genus $g$. Then we can construct and have PL {\rm (}piesewise smooth{\rm )} embeddings of $P^{\prime}$ into both some $3$-dimensional closed, connected and orientable manifold and some non-orientable one whose Heegaard genera are $g+l-1$.
\end{enumerate}
\end{MainThm}
This is a kind of explicit answers to Main Problem \ref{mprob:3}.
We also have the following two main theorems.
We define a {\it piesewise smooth} immersion in the fourth section. However, maybe we can guess the definition roughly.
\begin{MainThm}[Theorem \ref{thm:3} (\ref{thm:3.1})]
\label{mthm:3.1}
Let $P$ be a normal branched surface, $N$ an orientable surface with no boundary, and $c_0:P \rightarrow N$ a map born from an SSNS fold map. Then we have a {\rm piesewise smooth immersion} from $P$ into $Y:={\mathbb{R}}^3,S^3$. Furthermore, the piesewise smooth immersion is represented as $e_{N,Y} \circ i_{c_0,N}$ where the two maps are as follows.
\begin{enumerate}
\item $i_{c_0,N}:P \rightarrow N \times \mathbb{R}$ is a piesewise smooth immersion such that the composition with the canonical projection is $c_0$.
\item $e_{N,Y}$ is a piesewise smooth embedding into $Y$.
\end{enumerate}
\end{MainThm}
\begin{MainThm}[Theorem \ref{thm:3} (\ref{thm:3.2})]
\label{mthm:3.2}
In Main Theorem \ref{mthm:3.1}, suppose that the following conditions hold in addition.
\begin{itemize}
	\item $N$ is connected.
	\item $N-c_0(P)$ is not empty.
	\item $i_{c_0,N}:P \rightarrow N \times \mathbb{R}$ is a piesewise smooth embedding.
	\item The complementary set of the image $i_{c_0,N}(P)$ in $N \times \mathbb{R}$ is connected.

\end{itemize}
Then by applying Main Theorem \ref{mthm:2.1} with slightly revised arguments in the proof, we have a new piesewise smooth map $c:P \rightarrow N$ enjoying the following properties where we abuse the notation.
\begin{enumerate}
	\item We have a piesewise smooth embedding $i_{c,N}:P \rightarrow N \times \mathbb{R}$ such that the composition with the canonical projection is $c$.
	\item We have a piesewise smooth embedding $e_{N,Y} \circ i_{c,N}:P \rightarrow Y$.
\end{enumerate} 
\end{MainThm}

In the third and fourth sections, we discuss these results.
Related to the present study, we will also refer to systematic studies on global geometric properties on singularities of smooth maps between surfaces, shortly. Mainly, geometric and combinatorial studies such as \cite{yamamoto_t} by Takahiro Yamamoto are presented.

Remarks \ref{rem:1} and \ref{rem:2} are presented. They imply that we can generalize the present theory to cases where Reeb spaces are of general dimensions and that the present study is and can be one of the first explicit and systematic studies on new classes of simple fold maps. 

\section{Simple fold maps and their Reeb spaces.}
\begin{Def}
\label{def:3}
Let $m \geq n \geq 1$ be integers. Let $M$ be an $m$-dimensional closed and smooth manifold and $N$ be an $n$-dimensional smooth manifold with no boundary. A smooth map $f:M \rightarrow N$ is said to be a {\it simple fold} map if at each singular point $p$ $f$ is represented as
$$(x_1,\cdots,x_m) \rightarrow (x_1,\cdots,x_{n-1},{\Sigma}_{j=n}^{m-i(p)} {x_j}^2-{\Sigma}_{m-i(p)+1}^{m} {x_j}^2)$$
for suitable coordinates and a suitable integer $0 \leq i(p) \leq \frac{m-n+1}{2}$ and the restriction ${q_f} {\mid}_{S(f)}$ to the singular set is injective.
\end{Def}

Morse functions such that at mutually distinct singular points the values are always distinct account for a specific case or the case of $N=\mathbb{R}$ (where we need to and can perturb functions). Such functions are known to exist densely in the space of all smooth functions on any closed and smooth manifold where the space of the maps is endowed with the so-called {\it Whitney $C^{\infty}$ topology}. See \cite{golubitskyguillemin} as an introductory book on the singularity theory of differentiable maps including such notions and facts. 

So-called {\it special generic maps}, systematically studied in \cite{saeki2} for example, are also simple for example and the class of special generic maps contains Morse functions on homotopy spheres with two singular points and canonical projections of unit spheres.

In the present paper, we concentrate on cases $n= (\leq) 2$.
This class is a subclass of the class of {\it fold} maps, obtained by removing the condition on the injectivity of the map $q_f {\mid}_{S(f)}$.
For general or advanced systematic studies on fold maps, see \cite{golubitskyguillemin} before and see also \cite{eliashberg}, \cite{eliashberg2}, and \cite{saeki}, for example.

\begin{Prop}
\label{prop:1}
In Definition \ref{def:3}, $i(p)$ is unique.
The restrictions to $S(f)$ and the set of all singular points of any fixed index, shown to be {\rm (}$n-1${\rm )}-dimensional closed and smooth submanifolds with no boundaries for an arbitrary general fold map, are smooth immersions.
\end{Prop}

Fold maps such that the restrictions to the singular sets are embeddings are simple. {\it Round} fold maps, which are first defined by the author in \cite{kitazawa0.1,kitazawa0.2,kitazawa0.3}, satisfy this property. A {\it round} fold map is, in short, a fold map such that the restriction to the singular set is an embedding and that the image of the singular set is concentric.  

\begin{Def}
\label{def:4}
$i(p)$ is said to be the {\it index} of $p$ in Definition \ref{def:3}.
\end{Def}

Special generic maps are simple fold maps satisfying $i(p)=0$ for any singular point $p$.
We define the class of simple fold maps we concentrate on in the present paper.

\begin{Def}
\label{def:5}
A simple fold map in Definition \ref{def:3} is said to be {\it standard-spherical} if the following conditions are satisfied.
\begin{enumerate}
\item $m>n$.
\item The index of each singular point is always $0$ or $1$.
\item Preimages containing no singular points are always represented as disjoint unions of copies of the ($m-n$)-dimensional unit sphere.
\item If $m-n=1$ and there exists a singular point $p$ whose index is $i(p)=1$, then $M$ is orientable.
\end{enumerate}
\end{Def}

\begin{Prop}[E. g. \cite{kobayashisaeki, shiota}]
\label{prop:2}
For a simple fold map $f$ on an $m$-dimensional closed manifold into a surface with no boundary wherre $m \geq 3$, the Reeb space $W_f$ is a branched surface.
If the map is standard-spherical, then the Reeb space $W_f$ is a branched polyhedron where $q_f(S(f))=\sqcup C_j$ in Definition \ref{def:2}.
\end{Prop}

As the author knows, one of most general theorems of this type is given by \cite{shiota}. This says that for a smooth map of a very general class (, or a {\it Thom} map, which we do not define in the present paper,) the Reeb space is a polyhedron.

\cite{saeki3,saeki4,sakuma,yonebayashi} concentrate on standard-spherical or general simple fold maps into the plane, a surface, or a higher dimensional Euclidean space, and closed manifolds admitting such maps.
Especially, \cite{saeki4} is on $3$-dimensional closed, connected and orientable manifolds admitting simple fold maps (, which are as a result standard-spherical,) and presents a theorem that such a manifold admits a simple fold map into the plane if and only if it is a so-called {\it graph} manifold.
\cite{saeki3,yonebayashi} are closely related studies. 

\cite{turaev} presents the notion of so-called {\it shadows} of $3$-dimensional closed and orientable manifolds. They are in specific cases regarded as the Reeb spaces of simple fold maps into the plane on $3$-dimensional closed and orientable manifolds. They are also regarded as the Reeb spaces of more general fold maps into the plane on $3$-dimensional closed and orientable manifolds and may not be regarded as branched surfaces in Definition \ref{def:2}. These studies have born \cite{ishikawakoda} for example. 

PL embeddings of so-called {\it multibranched surfaces} into $3$-dimensional closed, connected, and orientable manifolds, are studied in \cite{matsuzakiozawa, munozozawa, ozawa} for example. Most of our branched surfaces are regarded as multibranched surfaces of some good classes.
 
Hereafter, we concentrate on the map $\bar{f}$ and continuous (PL) maps locally or globally regarded as such maps. 

\begin{Def}
\label{def:6}
For a simple fold map $f$ and the Reeb space $W_f$ in Proposition \ref{prop:2}, if the bundles whose fibers are PL homeomorphic to $K$ in Definition \ref{def:2} are always trivial, then $f$ is said to be {\it normal}. A branched surface satisfying the condition is also said to be {\it normal}.
\end{Def}

In the present paper, for simple fold maps, we mainly concentrate on standard-spherical and normal ones. We also call such a fold map an {\it SSNS fold} map.
 

The following presents fundamental facts or follows immediately from fundamental arguments in the theory of simple fold maps. We omit its proof. Consult section 2 of \cite{kobayashisaeki}, presented in Proposition \ref{prop:2}, for example. The proof of Proposition \ref{prop:2} is presented or we can prove the following proposition by applying arguments there in a suitable way.

\begin{Prop}
\label{prop:3}
For an SSNS fold map $f$, the following properties hold where $N$ is naturally regarded as a polyhedron and where the total spaces of trivial bundles are identified with the products of spaces in canonical ways.
\begin{enumerate}
\item For any sufficiently small open submanifold in $P-{\sqcup}_{j} C_j$ in Definition \ref{def:2}, the restriction of $\bar{f}$ there is a diffeomorphism onto an open submanifold in $N$ where the target is restricted there from $N$. Furthermore, the restriction to $P-{\sqcup}_{j} C_j$ is a smooth immersion. 
\item For $C_j$ in Definition \ref{def:2} such that $N(C_j)$ is the total space of a trivial PL bundle over $C_j$ whose fiber is PL homeomorphic to $[0,1]$ and each point $p \in C_j$ and a suitable small open neighborhood $U_{C,j,p} \subset C_j$ diffeomorphic to an open interval, ${\bar{f}} {\mid}_{U_{C,j,p} \times [0,1] \subset C_j \times [0,1]}$ is, by suitable coordinate transformations preserving the structures of the products of the manifolds, regarded as the product map of a piesewise smooth homeomorphism onto a closed interval from $[0,1]$ and a piesewise smooth homeomorphism from $U_{C,j,p}$ onto an open interval in the singular value set. Furthermore, the product of these two manifolds of the targets is a smooth and PL submanifold of $N$ and the manifold of the target of the restriction of $\bar{f}$ is restricted to the product manifold.
\item For $C_j$ in Definition \ref{def:2} such that $N(C_j)$ is the total space of a trivial PL bundle over $C_j$ whose fiber is PL homeomorphic to $K$ and each point $p \in C_j$ and a suitable small open neighborhood $U_{C,j,p} \subset C_j$ diffeomorphic to an open interval, ${\bar{f}} {\mid}_{U_{C,j,p} \times K \subset C_j \times K}$ is, by suitable coordinate transformations preserving the structures of the products of the manifolds, represented as the product map of a piesewise smooth map onto a closed interval and a piesewise smooth homeomorphism from $U_{C,j,p}$ onto an open interval in the singular value set. Furthermore, the product of these two manifolds of the targets is a smooth and PL submanifold of $N$ and the manifold of the target of the restriction of $\bar{f}$ is restricted to the product manifold. Furthermore, the piesewise smooth map onto a closed interval from $K$ is defined as the composition of the projection to $[-\frac{1}{2},1]$ obtained by considering the projection to the first component with a piesewise smooth map for scaling.
\end{enumerate}
\end{Prop}

The following is an important fact: "a pisewise smooth homeomorphism onto a closed interval from $[0,1]$" and "the piesewise smooth map onto a closed interval from $K$" in Proposition \ref{prop:3} are regarded as the compositions of natural smooth, piecewise smooth, or PL embeddings into ${\mathbb{R}}^2$, with canonical projections.

\begin{Def}
\label{def:7}
A continuous (pisewise smooth) map $c$ on a branched surface $P$ into a surface $N$ with no boundary is said to be {\it locally born from an SSNS fold map} if the following two hold where $N$ is regarded as a polyhedron.
\begin{itemize}
\item For each point $p \in P$, there exist a small regular neighborhood $N(p)$ of $p$, a regular neighborhood $N(c(p))$ of $c(p)$, a simple fold map $f_p$ on a closed manifold $M_p$ whose dimension is greater than $2$ into a surface $N_p$ with no boundary which is regarded as a polyhedron and subpolyhedra $N(p,M_p) \subset W_{f_p}$ and $N(c(p),N_p) \subset N_p$.
\item There exists a pair $({\Phi}_p:N(p) \rightarrow N(p,M_p),{\phi}_p:N(c(p)) \rightarrow N(c(p),N_p))$ of piesewise smooth homeomorphisms satisfying $\bar{f_p} {\mid}_{N(p,M_p)} \circ {\Phi}_p={\phi}_p \circ c {\mid}_{N(p)}$ for each $p \in P$ where for the maps $\bar{f_p} {\mid}_{N(p,M_p)}$ and $c {\mid}_{N(p)}$ the polyhedra of the targets are suitably restricted. 
\end{itemize}
As a specific case, $c$ is said to be {\it born from an SSNS fold map} if the following two hold.
\begin{itemize}
\item There exists a simple fold map $f$ on a closed manifold $M$ whose dimension is greater than $2$ into a surface $N$ with no boundary which is regarded as a polyhedron.
\item There exists a pair $(\Phi:P \rightarrow W_f,\phi:N \rightarrow N)$ of piesewise smooth homeomorphisms satisfying $\bar{f} \circ \Phi=\phi \circ c$.
\end{itemize}
\end{Def}

If the branched surface $P$ is normal, then $c$ is a map born from an SSNS fold map. If the branched surface $P$ is not normal, then $c$ is a map locally born from an SSNS fold map which is not one born from an SSNS fold map. 

In the present paper, for example in Theorems \ref{thm:1} and \ref{thm:2}, we can consider maps locally born from SSNS fold maps which are not maps born from SSNS fold maps. However, we essentially consider maps born from SSNS fold maps only.

A map locally born from an SSNS fold map which is not one born from an SSNS fold map is presented in section 7 of \cite{kobayashisaeki} for example. 

The following proposition works as a fundamental principle and for example, presented articles on (simple) fold maps will help us to understand this. We also present a sketch of a proof. 
\begin{Prop}
\label{prop:4}
Let $m \geq 3$ be an integer.
If a continuous {\rm (}piesewise smooth{\rm )} map $c:P \rightarrow N$ is born from an SSNS fold map and the restriction to ${\sqcup}_{j} C_j$ is a smooth immersion, then there exist an $m$-dimensional closed manifold $M$, an SSNS fold map $f:M \rightarrow N$ and a piesewise smooth homeomorphism $\phi:P \rightarrow W_f$ such that $\bar{f} \circ \phi=c$. Thus $c$ is born from an SSNS fold map.
\end{Prop}
\begin{proof}[{\rm (}A sketch of{\rm )} a proof.]
Take a small regular neighborhood $N(C_i)$ for each connected component $C_i$ in Definition \ref{def:2}.
We consider the product map of a Morse function satisfying either of the following two and the identity map on $C_i$ preserving the structures of the products for suitable coordinates.
\begin{itemize}
\item A Morse function on a copy of the unit disk $D^{m-1}$ represented by the form $(x_1,\cdots,x_{m-1}) \mapsto \pm {\Sigma}_{j=1}^{m-1} {x_j}^2+c_i$ for a real number $c_i$ for suitable coordinates. It is a so-called {\it height function} of a unit disk.
\item A Morse function on a smooth manifold diffeomorphic to one obtained in the following way.
\begin{itemize}
\item Take a copy of the unit disk $S^{m-1}$.
\item Remove the interiors of three disjointly and smoothly embedded copies of the ($m-1$)-dimensional unit disk $D^{m-1}$.
\end{itemize}
This function enjoys the following three properties.
\begin{itemize}
\item The preimage of an extremum is one or the disjoint union of two of the three connected components of the boundary.
\item The preimage of the remaining extremum is the disjoint union of the remaining connected components of the boundary. 
\item There exists exactly one singular point and it is in the interior of the manifold of the domain.
\end{itemize}
\end{itemize}
We consider the quotient map onto the Reeb space for the product map. By composing the map with a suitable piesewise smooth homeomorphism, we can have a desired map onto $N(C_i)$.
Over the complementary set of ${\sqcup}_i {\rm Int}\ N(C_i)$ in $P$, we can easily have a trivial smooth bundle whose fiber is diffeomorphic to $S^{m-2}$. We can glue the obtained maps in a canonical and suitable way and compose the resulting map with $c$. We have a desired map. 
\end{proof}

Hereafter, to make discussions simpler, we define a map $c$ in Proposition \ref{prop:4} as a map {\it born from an SSNS fold map}. Note that in fact we can discuss this proposition for a map locally born from an SSNS fold map which may not be born from an SSNS fold map and that we can define a similar notion for such a case. In short, we need to twist the smooth family of local Morse functions from the product map and we can do this. We also need to consider a smooth bundle which is not trivial over the complementary set of ${\sqcup}_i {\rm Int}\ N(C_i)$ in $P$  and we can do this. Except these arguments for example, we can argue similarly.

If $P$ is merely a surface in the definition, then $c$ is regarded as a smooth immersion into $N$ (if $\partial P$ is not empty) or a smooth map on a closed manifold into a closed surface with no singular points.   
Smooth immersions of curves and compact surfaces with boundaries into the plane or surfaces have been studied systematically and actively and there remain various difficult problems on the theory, seeming to be easy. We do not need advanced knowledge on the field. However, these studies can play important roles in studies related to the present study. We introduce several articles shortly. For example, \cite{whitney, whitney2} are classical important studies. Related studies are in \cite{chillingworth, kauffman, yamamoto_m} for example.

\section{Main Theorems and related studies.}
For a set $X$, ${\rm id}_X$ denotes the identity morphism.
Let $X$ be a compact polyhedron, $Y$ a polyhedron and $c:X \rightarrow Y$ a continuous map. Suppose that a {\it {\rm PL} isotopy}, defined as an isotopy considered in the piecewise smooth category ${\Phi}:Y \times [0,1] \rightarrow Y$ such that $\Phi(y,0)=y$, exists. We define ${\Phi} \circ (c \times {\rm id}_{[0,1]}):X \times [0,1] \rightarrow Y$ as a {\it {\rm P}-isotopy for} $c$.
For a branched surface $P$, a surface $N$ with no boundary, and a map $c:P \rightarrow N$ locally born from an SSNS fold map, let $B(c)$ denote the disjoint union $\sqcup C_j$ in Definition \ref{def:2}. This subset does not depend on such a map $c$ and depends only on the branched surface of the domain of the map (locally born from an SSNS fold map).
We present a fundamental method of constructing new branched surfaces and maps locally born from SSNS fold maps. 
 
We explain shortly about fundamental notions and notation on algeberaic topology. See \cite{hatcher} to study systematically.
 
Let $A$ be a commutative ring with a unique identity element different from the zero element and $X$ be a topological space. 
Let $X^{\prime} \subset X$ be a subspace.

$H_j(X,X^{\prime};A)$ denotes the {\it $j$-th homology group} of the pair $(X,X^{\prime})$ whose {\it coefficient ring} is $A$. 
If $X^{\prime}$ is empty, then $H_j(X;A)$ is also used instead and it is the {\it $j$-th homology group} of $X$ whose {\it coefficient ring} is $A$.
$H^j(X,X^{\prime};A)$ denotes the {\it $j$-th cohomology group} of the pair $(X,X^{\prime})$ whose {\it coefficient ring} is $A$. If $X^{\prime}$ is empty, then $H^j(X;A)$ is also used instead and it is the {\it $j$-th cohomology group} of $X$ whose {\it coefficient ring} is $A$. $H^{\ast}(X;A):={\oplus}_{j} H^j(X;A)$ denotes the direct sum of the $j$-th cohomology group $H^j(X;A)$ for all integers $j$. This has the structure of a graded ring where each product is given by the so-called {\it cup product}. This is the cohomology ring of $X$ whose {\it coefficient ring} is $A$.

The fundamental group of an arcwise-connected topological space $X$ is defined and denoted by ${\pi}_1(X)$.

For two pairs of topological spaces $(X_1,{X_1}^{\prime})$ and $(X_2,{X_2}^{\prime})$ satisfying the relation ${X_i}^{\prime} \subset X_i$ and a continuous map $c:X_1 \rightarrow X_2$ satisfying $c({X_1}^{\prime}) \subset {X_2}^{\prime}$, we can define the {\it canonical homomorphisms} $c_{\ast}:H_j(X_1;{X_1}^{\prime}) \rightarrow H_j(X_2;{X_2}^{\prime})$ and $c^{\ast}:H^j(X_2;{X_2}^{\prime}) \rightarrow H^j(X_1;{X_1}^{\prime})$
{\it induced by $c$}. If $X_1$ and $X_2$ are arcwise-connected and ${X_1}^{\prime}$ and ${X_2}^{\prime}$ are empty, then we can define the {\it canononical homomorphism} $c_{\ast}:{\pi}_1(X_1) \rightarrow {\pi}_1(X_2)$ {\it induced by $c$}.

For any homology class $h_c \in H_k(X;A)$ we cannot represent by the form $r{h_c}^{\prime}$ for any element $r$ which is not a unit in the coefficient ring $A$ and any class ${h_c}^{\prime}$ which is not the zero element and an basis $\mathcal{B}$ of the submodule generated by all free elements of $H_k(X;A)$ satisfying $h_c \in \mathcal{B}$, we have the element ${h_c}^{\mathcal{B},\ast} \in H^k(X;A)$ enjoying the following properties uniquely. 
\begin{itemize}
\item ${h_c}^{\mathcal{B},\ast}(h_c)=1$ where $1$ denotes the identity element.
\item ${h_c}^{\mathcal{B},\ast}(\mathcal{B}-\{h_c\})$ is the trivial group.
\end{itemize}
 This is called the {\it cohomology dual} ${h_c}^{\mathcal{B},\ast}$ to $h_c$ on the {\it basis $\mathcal{B}$}. We omit the exposition on the basis in considerable cases where we can guess the basis naturally. 

It is well known that for a compact, connected and oriented manifold $X$ whose boundary $\partial X$ may not be empty, we have a unique element $u_X \in H_{\dim X}(X,\partial X;A)$ which is a generator of the group, isomorphic to $A$ as a module over $A$. This is the {\it fundamental class} of $X$ (where the {\it coefficient ring} is $A$). For the element, we do not need an orientation where the orders of elements there are at most $2$.

We also need some theorems and methods such as Mayer-Vietoris sequences and Seifert van Kampen theorem. We omit precise expositions.

\begin{Thm}
\label{thm:1}
Let $P$ be a  branched surface, $N$ a surface with no boundary, and $c:P \rightarrow N$ a map {\rm (}locally{\rm )} born from an SSNS fold map.
Let $\{T_j\}_j$ be a family of finitely many circles which are disjointly embedded in $P-B(c)$ smoothly and regarded as subpolyhedra of $P$. Let the size of the family $\{T_j\}_j$ be $l>0$.
Let ${\bigcup}_{j} c(T_j)$ be the image of the boundary of a compact and connected surface $S_C$ smoothly immersed into $N$.
Then by attaching a surface homeomorphic to $S_C$ along ${\sqcup}_j T_j$ on the boundary by a PL homeomorphism, a piesewise homeomorphism, or a diffeomorphism, we have a new {\rm (}normal{\rm )} branched surface $P^{\prime}$ and a map $c^{\prime}:P^{\prime} \rightarrow N$ {\rm (}resp. locally{\rm )} born from an SSNS fold map.  

Furthermore, we have the following two where $A$ is a principal ideal domain having a unique identity element different from the zero element.
\begin{enumerate}
\item 
\label{thm:1.1}
Let the value of the homomorphism induced by the inclusion of the {\rm (}suitably oriented{\rm )} circle $T_j$ at the fundamental class is the zero element of $H_1(P;A)$. Assume that $P$ is connected. Assume also that $S_C$ is orientable. Then we have the following information on the homology groups, the cohomology groups. the cohomology rings and the fundamental groups.
\begin{enumerate}
\item
\label{thm:1.1.1}
 $H_1(P^{\prime};A)$ is isomorphic to and identified with the direct sum of $H_1(P;A)$, $A^{l-1}$, and $H_1(S_{C,0};A)$ where $S_{C,0}$ denotes a closed and connected surface obtained by attaching copies of the $2$-dimensional unit disk to $S_{C}$ by a diffeomorphism between the boundaries. Furthermore, this cohomology group is free as a module over $A$. 
\item
\label{thm:1.1.2}
 $H_2(P^{\prime};A)$ is isomorphic to and identified with the direct sum of $H_2(P;A)$ and $A$.
\item
\label{thm:1.1.3}
 The cohomology group $H^{1}(P^{\prime};A)$ is isomorphic to and identified with the direct sum of $H^{1}(P;A)$, $H^{1}(S_{C,0};A)$, and $A^{l-1}$ where $A^{l_0}$ denote the free module over $A$ of rank $l_0 \geq 0$. If $H_1(P^{\prime};A)$ is free, then the cohomology group $H^2(P^{\prime};A)$ is isomorphic to the direct sum of $H^2(P;A)$ and $A$. Furthermore, we have the following facts on the cohomology rings and the fundamental groups.
\begin{enumerate}
\item
\label{thm:1.1.3.1}
 $H^{1}(P;A) \oplus \{0\} \subset H^{1}(P^{\prime};A)$ generates a subalgebra of the cohomology ring $H^{\ast}(P^{\prime};A)$ and this is isomorphic to the cohomology ring $H^{\ast}(P;A)$ where $\{0\}$ denotes the trivial group.
\item 
\label{thm:1.1.3.2}
$\{0\} \oplus H^{1}(S_{C,0};A) \oplus \{0\} \subset H^{1}(P^{\prime};A)$ generates a subalgebra of the cohomology ring $H^{\ast}(P^{\prime};A)$ and this is isomorphic to the cohomology ring $H^{\ast}(S_{C,0};A)$ where $\{0\}$ denotes the trivial group.
\item
\label{thm:1.1.3.3}
 The cup product of each element in $\{0\} \oplus H^{1}(S_{C,0};A) \oplus A^{l-1} \subset H^{1}(P^{\prime};A)$ and each element in $\{0\} \oplus A^{l-1} \subset H^{1}(P^{\prime};A)$ is always the zero element $H^{2}(P^{\prime};A)$ where $\{0\}$ denotes the trivial group.
\item
\label{thm:1.1.3.4}
 Assume also that the value of the homomorphism induced by the inclusion of the circle $T_j$ at a generator of the fundamental group, isomorphic to $\mathbb{Z}$, is the zero element of ${\pi}_1(P)$.
The fundamental group of $P^{\prime}$ is isomorphic to the free product of the group ${\pi}_1(P)$ and $l-1$ copies of the group $\mathbb{Z}$, which is the group of all integers.
\end{enumerate}
\end{enumerate}
\item
\label{thm:1.2}
 If $A$ consists of elements whose orders are at most $2$, then we do not need to assume that $S_C$ is orientable in the previous situation.
\end{enumerate} 
\end{Thm}
\begin{proof}
$S_C$ is smoothly immersed into $N$ and the image of the boundary is ${\bigcup}_{j} c(T_j)$. This naturally induces a new piesewise smooth map $c^{\prime}$ on the new branched surface $P^{\prime}$ which is also a map locally born from an SSNS fold map. We can construct this in such a way that $B(c^{\prime})$ is regarded as the disjoint union of $B(c)$ and ${\bigcup}_{j} c(T_j)$.
 If the given map is a map born from an SSNS fold map, then we can have a map as a map born from an SSNS fold map. 
 We can know this easily.
 This completes the proof of the fact before (\ref{thm:1.1}).

We prove (\ref{thm:1.1}). 
We have a Mayer-Vietoris sequence

$$\rightarrow H_i({\sqcup}_j T_j;A) \rightarrow H_i(P;A) \oplus H_i(S_C;A) \rightarrow H_i(P^{\prime};A) \rightarrow$$

and we present a sketch of important arguments.

The quotient group $H_2(P^{\prime};A)/H_2(P;A)$ can be defined naturally by the assumption and it is isomorphic to $A$. For this, note that $H_2(S_C;A)$ and $H_2({\sqcup}_j T_j;A)$ are the trivial groups, that the value of the canonical homomorphism induced by the inclusion of the oriented circle $T_j$ at the fundamental class is the zero element of $H_1(P;A)$, and that the image of the homomorphism from $H_1({\sqcup}_j T_j;A)$ into $H_1(P;A) \oplus H_1(S_C;A)$ consists of elements of the form $(0,t)$ and of rank $l-1$ where $0$ is the zero element. Thus $H_2(P^{\prime};A)$ is isomorphic to the direct sum of $H_2(P;A)$ and $A$. 
More precisely, the summand $H_2(P;A)$ is regarded as a group generated by a set of elements in $H_2(P;A)$ mapped by the canonical homomorphism induced by the inclusion of $P$ into $P^{\prime}$, regarding $P \subset P^{\prime}$ naturally. 
For $H_1(P^{\prime};A)$, from these arguments with the structures of the homomorphisms for $i=0$, we have a desired result.

This completes the proof of (\ref{thm:1.1.1}) and (\ref{thm:1.1.2}).

We prove (\ref{thm:1.1.3}) except (\ref{thm:1.1.3.4}).

For the cohomology groups, we have desired results by fundamental arguments such as ones using universal coefficient theorem.

First we add an exposition on $H_1(P^{\prime};A)$.
The summand $H_1(P;A)$ is regarded as the image of the canonical homomorphism from $H_1(P;A)$ into $H_1(P^{\prime};A)$ induced by the inclusion of $P$ into $P^{\prime}$. We have the following facts.

\begin{itemize}
	\item The summand $H_1(S_{C,0};A)$ is a module generated by the elements which are the values of the homomorphisms induced by the canonical inclusions of (suitably oriented) circles embedded smoothly in the interior ${\rm Int}\ S_{C} \subset S_{C} \subset S_{C,0}$ of the surface $S_{C}$ in $S_{C,0}$. 
	\item The summand $A^{l-1}$ is generated by exactly $l-1$ elements which are the values of the homomorphisms induced by the canonical inclusions of (suitably oriented) suitable circles embedded as subpolyhedra in $P^{\prime}$ at the fundamental classes. Furthermore, we can choose these $l-1$ circles here in such a way that each of the $l-1$ circles and the previous circles in ${\rm Int}\ S_{C}$ do not intersect.
	
\end{itemize}

  For the cup products, consider two homology classes the cohomology duals to which we can define and we can find suitable piesewise smooth maps from oriented circles such that the values of the induced homomorphisms at the fundamental classes are these homology classes. We can investigate the cup products of the cohomology duals by observing homology classes in such a way. This completes the proof of (\ref{thm:1.1.3}) except (\ref{thm:1.1.3.4}).

We prove (\ref{thm:1.1.3.4}). Here $P^{\prime}$ is regarded as a polyhedron which is simple homotopy equivalent to a polyhedron obtained by choosing exactly $l$ points of $P$ and those of $S_{C,0}$, $l$ pairs of points each of which is a pair of a point in $P$ and one in $S_{C,0}$ and identifying the two points for each pair.

By Seifert van Kampen theorem, this completes the proof of (\ref{thm:1.1.3.4}).

These arguments complete the proof of (\ref{thm:1.1}).
We can prove (\ref{thm:1.2}) similarly.

This completes the proof.

\end{proof}
\begin{Ex}
\label{ex:1}
In the assumption of Theorem \ref{thm:1}, if the immersion of $S_C$ is an embedding and the image is diffeomorphic to a compact, connected, and orientable surface of genus $0$ and in the interior of a smoothly embedded copy of the $2$-dimensional unit disk in $N$, then we can apply Theorem \ref{thm:1}. If the additional conditions are satisfied, then we reach the additional results (\ref{thm:1.1}) and (\ref{thm:1.2}).
\end{Ex}

Hereafter, various principles such as approximations by PL or piecewise smooth maps of continuous maps are well-known. However, we do not explain precisely here. It does not affect our understanding much if we do not understand these notions and principles so much.

\begin{Thm}
\label{thm:2}
Let $P$ be a branched surface, $N$ a surface with no boundary, and $c_0:P \rightarrow N$ a map {\rm (}locally{\rm )} born from an SSNS fold map.
Let $\{D_j\}_j$ be a family of finitely many copies of the $2$-dimensional unit disk which are disjointly embedded in $P-B(c_0)$ smoothly and regarded as subpolyhedra of $P$.
We also assume that $N$ is connected and that $N-c_0(P)$ is not empty. We have the following two.
\begin{enumerate}
	\item \label{thm:2.1}
 By a suitable piesewise smooth homotopy $F_c$ from $c_0$ to a new map $c:P \rightarrow N$ and apply Theorem \ref{thm:1} by setting $T_j:=F_c(\partial D_j \times \{1\})$ and $S_C$ as a compact, connected and orientable surface of genus $0$.
Furtherore, if $P$ is connected, then we can also apply additional statements on homology groups, cohomology groups and rings, fundamental groups in Theorem \ref{thm:1}.
\item \label{thm:2.2}
We also have the following two.
\begin{enumerate}
		\item Assume that we can embed $P$ as a PL or smooth submanifold in a $3$-dimensional closed and connected manifold of Heegaard genus $g$. Then we can construct and have a PL {\rm (}piesewise smooth{\rm )} embedding of $P^{\prime}$ into one whose Heegaard genus is $g+l-1$.
	\item Assume that we can embed $P$ as a PL or smooth submanifold in a $3$-dimensional closed, connected and orientable manifold of Heegaard genus $g$. Then we can construct and have PL {\rm (}piesewise smooth{\rm )} embeddings of $P^{\prime}$ into both some $3$-dimensional closed, connected and orientable manifold and some non-orientable one whose Heegaard genera are $g+l-1$.
\end{enumerate}

\end{enumerate}
\end{Thm}
\begin{proof}
	We prove (\ref{thm:2.1}).
	
By a suitable piesewise smooth homotopy from $c_0$ to a new map $c_1:P \rightarrow N$ born from an SSNS fold map, we can obtain a situation satisfying the following three. We deform the map around a small regular neighborhood of each $D_j$ by this homotopy. 
\begin{itemize}
\item $c_1 {\mid}_{D_j}$ is a PL embedding.
\item $c_1(D_j)$ in the family $\{c_1(D_j)\}_j$ are mutually disjoint.
\item $N-c_1(P)$ is not empty. 
\end{itemize}
$N$ is assumed to be connected. This enables us to obtain $c_1$ and take a smooth map $t_j$ which is an embedding of a closed interval and connects a point in $c_1({\rm Int}\ D_j)$ to a point $p_j$ in a connected component of $N-c_1(P)$ in such a way that the following properties are enjoyed. See also Remark \ref{rem:0} later and we omit related precise principles.
\begin{itemize}
	\item The images of the maps in $\{t_j\}_j$ are mutually disjoint for distinct $j=j_1$ and $j=j_2$. 
\item If we restrict $c_1$ to $B(c_1)$, then it is a smooth immersion. This is also included in the revised definition of a map (locally) born from an SSNS fold map, presented just after Proposition \ref{prop:4}.  
\item We can construct a smooth immersion represented as the disjoint union of the smooth immersion $c_1 {\mid}_{B(c_1)}$ just before and all maps in $\{t_j\}_j$. This smooth immersion enjoys the following two.
\begin{itemize}
\item The number of points in the preimage of a point in $N$ is at most two. The number of points in $N$ whose preimages have exactly two points is finite.
\item The intersection of the images of the one or two differentials at the points in the preimage is $1$-dimensional if the preimage consists of exactly one point and $0$-dimensional if the preimage consists of exactly two points.
\end{itemize}
\end{itemize}
For the map $c_1$, we take the connected component containing a point in $D_j$ of the preimage $P_j$ of a small regular neighborhood of the image of each curve $t_j$.
Hereafter, each homotopy used here for deformations is chosen as a P-isotopy where the polyhedron of the domain is restricted to $D_j \sqcup P_j$ (or its small regular neighborhood).
 
By a suitable piesewise smooth homotopy from $c_1$ to a suitable new map $c_2:P \rightarrow N$ born from an SSNS fold map, we can obtain a situation enjoying the following two. 
\begin{itemize}
\item $c_2 {\mid}_{D_j}$ is a piesewise smooth embedding.
\item Distinct sets in the family $\{c_2(D_j \bigcup P_j)\}_j$ are mutually disjoint and in the interior of some small smoothly embedded copy $D_{P,1}$ of the $2$-dimensional unit disk in $N-c_1(P)$.
\end{itemize}

We can see that using a suitable piesewise smooth homotopy from $c_2$ to a suitable new map $c_3:P \rightarrow N$ born from an SSNS fold map if we need, we can obtain a situation enjoying the following four.
\begin{itemize}
\item $c_3 {\mid}_{D_j}$ is a piesewise smooth embedding.
\item Distinct sets in the family $\{c_3(D_j \bigcup P_j)\}_j$ are mutually disjoint.
\item We can choose a family $\{D_{P,2,j}\}_j$ of copies of the $2$-dimensional unit disk $D^2$ embedded smoothly and disjointly in the interior ${\rm Int}\ D_{P,1}$ of $D_{P,1}$ as sufficiently small submanifolds in such a way that the image $c_3(D_j \bigcup P_j)$ before is a subspace of the interior ${\rm Int}\ D_{P,2,j}$ of $D_{P,2,j}$.
\item We can take $E_j$ as a suitable small neighborhood of $D_j \bigcup P_j$ and a 2-dimensional compact and connected subpolyhedron in $P$ in such a way that the relations $D_j \bigcup P_j \subset {\rm Int}\ E_j \subset E_j \subset {c_3}^{-1}({\rm Int}\ D_{P,1})$, $c_3(D_j \bigcup P_j) \subset {\rm Int}\ D_{P,2,j} \subset D_{P,2,j} \subset c_3(E_j)$ and ${c_3}^{-1}(D_{P,2,j}) \subset E_j$ hold. 
\end{itemize}

See also FIGURE \ref{fig:1} for the image of $c_3$ and some subspaces of $N$ discussed here.
\begin{figure}
\includegraphics[width=35mm]{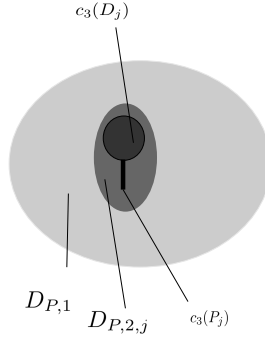}
\caption{The image of $c_3$ and several important subspaces of $N$.}
\label{fig:1}
\end{figure}

We can see that using a suitable piesewise smooth homotopy from $c_3$ to a new suitable map $c:P \rightarrow N$ born from an SSNS fold map, we can obtain a piesewise smooth homotopy $F_c$ of $c_0$ to $c$ as a result and a situation of Example \ref{ex:1} such that $T_j:=F_c(\partial D_j \times \{1\})$ and that the image of $S_C$ is $F_c(D_{j_0} \times \{1\})-{\sqcup}_{j \in J-\{j_0\}} F_c(D_j \times \{1\})$: $J$ denotes the set of all $j$, $j_0$ is a suitable element of $J$ and we can also obtain the situation such that for the suitably chosen element $j_0 \in J$, subsets in the family $\{F_c(D_j \times \{1\})\}_{j \in J-\{j_0\}}$ are mutually disjoint and also subsets of $F_c(D_{j_0} \times \{1\})$. 

We can show the last fact, obtained by applying Theorem \ref{thm:1}, immediately. This completes the proof of (\ref{thm:2.1}).

We prove (\ref{thm:2.2}).
For a $3$-dimensional closed and connected manifold $Y$ whose Heegaard genus is $g$ and which we have a PL (piesewise smooth) embedding into $Y$, we can choose $l>0$ small copies of the $3$-dimensional unit disk embedded disjointly and smoothly each of which contains $D_j$ in the following manners.
\begin{itemize}
	\item $D_j$ is embedded smoothly in one of the copies of the $3$-dimensional unit disk $D^3$, denoted by ${D^3}_j$. 
	\item The boundary $\partial D_j$ is embedded in the boundary of the copy ${D^3}_j$ of the $3$-dimensional unit disk and the interior ${\rm Int}\ D_j$ is embedded in the interior ${\rm Int}\ {D^3}_j$ of the copy ${D^3}_j$ of the $3$-dimensional unit disk. In other words, it is {\it embedded properly}.
	\item The complementary set $P-{\rm Int}\ D_j$ of $D_j$ in $P$ and the copy ${D^3}_j$ of the $3$-dimensional unit disk here do not intersect.
\end{itemize}
First, by a P-isotopy we can deform the original embedding of $P$ into $Y$ and move the copies of the $2$-dimensional unit disk in $\{D_j\}_j$ in such a way that the following conditions hold.
\begin{itemize}
	\item The interior of each $D_j$ is moved into the complementary set of the disjoint union ${\sqcup}_j {D^3}_j$ of the $l$ chosen copies of the $3$-dimensional unit disk in $Y$ by the deformation by the P-isotopy.
	\item All points in the complementary set of the interior ${\rm Int}\ D_j$ of each $D_j$ are fixed in the deformation by the P-isotopy.
\end{itemize}
After that remove the interiors of the copies of the $3$-dimensional unit disk in $Y$ or ${\rm Int}\ {\sqcup}_j {D^3}_j$ and attach a compact manifold ${S^3}_{-l}$ diffeomorphic to a manifold obtained by removing the interiors of another $l$ smoothly and disjointly embedded copies of the $3$-dimensional unit disk from a copy of the $3$-dimensional unit sphere $S^3$ instead. This yields a $3$-dimensional closed and connected manifold whose Heegaard genus is $g+l-1$, which is due to fundamental properties on Heegaard genera. If $Y$ is non-orinetable, then the resulting manifold is also non-orientable by fundamental arguments on the topologies of these manifolds. We can also see that $S_C$ here is embedded smoothly in the manifold ${S^3}_{-l}$ obeying the following rules.  
 
\begin{itemize}
	\item The boundary of $S_C$ is embedded in the boundary of ${S^3}_{-l}$.
	\item The interior of $S_C$ is embedded in the inteiror of ${S^3}_{-l}$.
	\item Distinct connected components of the boundary of $S_C$ are in distinct connected components of the boundary of ${S^3}_{-l}$.
	%
\end{itemize}

We present a well-known argument on (PL or piesewise smooth) homeomorphisms and diffeomorphisms on $2$-dimensional spheres and circles embedded there. 

Consider a copy of the $2$-dimensional unit sphere $S^2$ and a circle embedded 
here. If it is embedded as a smooth submanifold, then a diffeomorphism mapping the circle onto the circle itself extends to a diffeomorphism on the copy of the $2$-dimensional unit sphere $S^2$. If it is embedded by a piesewise smooth embedding, then a piesewise smooth homeomorphism mapping the circle onto the circle itself extends to a piesewise smooth homeomorphism on the copy of the $2$-dimensional unit sphere $S^2$. If it is embedded by a PL embedding, then a PL homeomorphism mapping the circle onto the circle itself extends to a piesewise smooth homeomorphism on the copy of the $2$-dimensional unit sphere $S^2$.
Furthermore, in all cases, we can extend the diffeomorphism, the piesewise homeomorphism, or the PL homeomorphism to both orientation preserving and reversing ones.

This argument with some fundamental arguments on the topologies of the manifolds we consider completes the proof of (\ref{thm:2.2}).
\end{proof}



The following notion is not defined in published articles as the author knows. Related to this, Makoto Ozawa informally introduced the author that we can define {\it PL immersions} for multibranched surfaces. We define {\it piesewise smooth immersions}, which are morphisms in the piesewise smooth category.
\begin{Def}
\label{def:8}
For a piesewise smooth map $c_0$ on a branched surface $P$ into a $3$-dimensional manifold $X$ with no boundary, it is said to be a {\it piesewise smooth immersion} if
there exists a homotopy between $c_0$ and a map $c:P \rightarrow X$, satisfying the following four conditions.
\begin{enumerate}
\item \label{def:8.1} If we restrict the homotopy to $(P-{\sqcup}_j C_j) \times [0,1]=(P-B(c)) \times [0,1]$ in Definition \ref{def:2}, then it is a P-isotopy and the restriction of this to $(P-B(c)) \times \{t\}$ for each $t \in [0,1]$ is a smooth immersion.
\item \label{def:8.2}
If we restrict the homotopy to ${\sqcup}_j C_j \times [0,1]= B(c)\times [0,1]$ in Definition \ref{def:2}, then it is a P-isotopy and the restriction of this to  $B(c) \times \{t\}$ for each $t \in [0,1]$ is a smooth embedding.  
\item \label{def:8.3}
For the immersion $c_{{\rm i},2}$ obtained in {\rm (}\ref{def:8.1}{\rm )} as the composition of the map for the natural identification mapping $x \in P-B(c)$ to $(x,1)$ with the restriction of the P-isotopy to $(P-B(c)) \times \{1\}$ and any point $p$ in the image, the following two are satisfied.
\begin{enumerate}
\item \label{def:8.3.1}
 The preimage consists of at most two points. 
\item \label{def:8.3.2}
The intersection of the images of the one or two differentials at the points in the preimage is $2$-dimensional if the preimage consists of exactly one point and $1$-dimensional if the preimage consists of exactly two points.
\end{enumerate}
\item \label{def:8.4}
 For the immersion $c_{{\rm i},2}$ and the embedding $c_{{\rm e},1}:B(c) \rightarrow X$ obtained in {\rm (}\ref{def:8.2}{\rm )} in a similar way, the following two are satisfied.
\begin{enumerate}
\item \label{def:8.4.1}
 The number of points in $X$ which is in the images of the both maps is finite. Furthermore, the preimage consists of exactly two points for each of such points for the map represented as the disjoint union $c_{{\rm i,}2} \sqcup c_{{\rm e},1}$. 
\item \label{def:8.4.2}
For each point in the images of the both maps, the intersection of the images of the differentials at the points in the preimage is $0$-dimensional.
\end{enumerate}
\end{enumerate}
\end{Def}
\begin{Rem}
\label{rem:0}
The conditions (\ref{def:8.3}) and (\ref{def:8.4}) of Definition \ref{def:8} respect so-called "transwersality (transversalities)" of submanifolds (in the PL, piecewise smooth, or smooth category). Similar conditions appear in the proof of Theorem \ref{thm:2} for example. Properties of this type hold in general where we discuss our problems: they are so-called {\it generic} ones. For example, \cite{golubitskyguillemin} explains about fundamental notions and properties on transversality (transversalities) of submanifolds and smooth maps and {\it generic} smooth maps.
\end{Rem}

By respecting our arguments, we can also obtain the following theorem.

\begin{Prop}
\label{prop:5}
Let $P$ be a normal branched surface, $N$ a surface with no boundary, and $c:P \rightarrow N$ a map born from an SSNS fold map. Then $c$ is represented as the composition of a piesewise smooth immersion into $N \times \mathbb{R}$ with the canonical projection to $N$.
\end{Prop}
\begin{proof}
 Let $C_j$ denote each connected component of $B(c)$ and let $N(C_j)$ denote its suitable small regular neighborhood. 
 Note that the bundle $N(C_j)$ over $C_j$ is a trivial PL bundle since $P$ is normal.
 We can construct a PL (piesewise smooth) embedding $e_{B(c)}:{\sqcup}_j N(C_j) \rightarrow {\sqcup}_j N(C_j) \times \mathbb{R} \subset N \times \mathbb{R}$ such that the composition with the canonical projection to $N$ is $c {\mid}_{{\sqcup}_j N(C_j)}$. More precisely, we can argue in this way due to comments just after Proposition \ref{prop:3} for example.
For each connected component $P_R$ of $N-{\rm Int}\ {\sqcup}_j N(C_j)$, we have a similar piesewise smooth immersion $e_R:P_R \rightarrow P_R \times \mathbb{R}$. More precisely, we can also argue in this way due to comments just after Proposition \ref{prop:3} for example. By gluing them together in a suitable way, we have a desired piesewise smooth immersion. Note also that via fundamental arguments on transversality, discussed in Remark \ref{rem:0} shortly, we have a piesewise smooth immersion as in Definition \ref{def:8}.
\end{proof}
\begin{Thm}
\label{thm:3}
Let $P$ be a normal branched surface, $N$ an orientable surface with no boundary, and $c_0:P \rightarrow N$ a map born from an SSNS fold map. Then we have a piesewise smooth immersion from $P$ into $Y:={\mathbb{R}}^3,S^3$.
\begin{enumerate}
	\item \label{thm:3.1}
	 Furthermore, the piesewise smooth immersion is represented as $e_{N,Y} \circ i_{c_0,N}$ where the two maps are as follows.
\begin{enumerate}
\item $i_{c_0,N}:P \rightarrow N \times \mathbb{R}$ is a piesewise smooth immersion such that the composition with the canonical projection is $c_0$.
\item $e_{N,Y}$ is a piesewise smooth embedding into $Y$.
\end{enumerate}
\item \label{thm:3.2}
In {\rm (}\ref{thm:3.1}{\rm )}, suppose that the following conditions hold in addition.
\begin{itemize}
	\item $N$ is connected.
		\item $N-c_0(P)$ is not empty.
	\item $i_{c_0,N}:P \rightarrow N \times \mathbb{R}$ is a piesewise smooth embedding.
	\item The complementary set of the image $i_{c_0,N}(P)$ in $N \times \mathbb{R}$ is connected.
	\end{itemize}
	 Then by applying Theorem \ref{thm:2} with slightly revised arguments in the proof, we have a new piesewise smooth map $c:P \rightarrow N$ enjoying the following properties. 
	 \begin{enumerate}
	 	\item We have a piesewise smooth embedding $i_{c,N}:P \rightarrow N \times \mathbb{R}$ such that the composition with the canonical projection is $c$.
	 	\item We have a piesewise smooth embedding $e_{N,Y} \circ i_{c,N}:P \rightarrow Y$.
	 \end{enumerate} 
\end{enumerate}
\end{Thm}

\begin{proof}
We prove (\ref{thm:3.1}).

$N \times \mathbb{R}$ can be embedded as a smooth submanifold in $Y$. Moreover, this is seen as a tubular neighborhood of a smooth submanifold diffeomorphic to $N$.
By enlarging the manifold of the target, which is diffeomorphic to $N \times \mathbb{R}$, in Proposition \ref{prop:5}, to $Y$, we have a desired result. 

We prove (\ref{thm:3.2}). Our proof of this applies transversality, discussed in Remark \ref{rem:0}, implicitly. We also abuse the notation in (the proof of) Theorem \ref{thm:2} for example.

$N$ and $N \times \mathbb{R}$ are assumed to be connected. $(N \times \mathbb{R})-i_{c_0,N}(P)$ is assumed to be connected where $i_{c_0,N}(P)$ is the image of the piesewise smooth embedding $i_{c_0,N}$. $c_0$ is assumed to be not surjective.
 Due to these conditions, we can obtain a piesewise smooth homotopy $F_C$ in the proof of Theorem \ref{thm:2} represented as the composition of a P-isotopy from $i_{c_0,N}$ to another piesewise smooth embedding $i_{c,N}$ with the canonical projection from $N \times \mathbb{R} \times [0,1]$ to $N \times [0,1]$, where we choose the family $\{t_j\}_j$ of the maps in a slightly revised way. We can choose the family enjoying the following properties.
\begin{itemize}
	\item Each map $t_j$ is a smooth embedding of a closed interval.
	\item For distinct maps in the family, the images may intersect only in the interiors (satisfying the "transversality").
\end{itemize} 
This gives a desired piesewise smooth map $c:P \rightarrow N$ represented as the composition of some piesewise smooth embedding $i_{c,N}$ with the canonical projection.
\end{proof}

We present our main motivation of the studies. This is also presented in the first section a little and will be presented in the last section a little. This new study is mainly motivated by studies of the author on global topologies of Reeb spaces of explicit fold maps and also motivated by geometric and combinatorial studies on smooth maps on closed surfaces into the plane or surfaces with no boundaries, mainly ones in \cite{yamamoto_t} and related articles by Takahiro Yamamoto, which are also references in \cite{yamamoto_t}. \cite{pignoni} is also a closely related work. Yamamoto introduced and obtained a family of fundamental operations of changing surfaces and maps on surfaces into surfaces. He has also studied patterns of tuples of numbers consisting of genera of surfaces (of domains), mapping degrees of maps between closed, connected and oriented surfaces, numbers of connected components of singular sets, consisting of circles, numbers of singular points defined as {\it cusp} points, and numbers of so-called {\it double points} or {\it nodes} for the images of singular sets, for example. In other words, he has been challenging a kind of studies on the so-called geography of {\it generic} smooth maps on closed surfaces into the plane or surfaces with no boundaries. This has motivated us to study the geography of branched surfaces and that of maps born from SSNS fold maps. Branched surfaces are in general not being manifolds and so-called {\it non-manifold} points seem to play similar roles as singular points play in cases for smooth maps on comapct surfaces into surfaces with no boundaries.

General algebraic topological or differential topological theory on smooth maps on closed manifolds whose dimensions are greater than or equal to $2$ into the plane are also presented in \cite{chess} and \cite{quine}, for example. \cite{yonebayashi} presents some related results such as new proofs of formula in \cite{chess} for numbers of double points for the images of the restrictions to the singular sets for simple fold maps on $3$-dimensional closed and orientable manifolds into surfaces.

Last we present additional remarks. 

\begin{Rem}
\label{rem:1}
Some of arguments and results in the present paper can be generalized for {\it branched manifolds}, defined in \cite{kitazawa11}, based on existing studies on (generic) smooth maps in the generalized cases. Generalizations are in some cases routine works and may be difficult in some cases. \cite{kitazawa11} concentrates on global topologies of Reeb spaces in cases where the dimensions of the manifolds are general.
\end{Rem}
\begin{Rem}\label{rem:2}
On global topologies of branched surfaces and maps born from SSNS fold maps, \cite{kitazawa11} explains about existing related studies shortly. As presented there, \cite{kobayashisaeki} and \cite{naoe} present or implicitly imply that branched surfaces resembling disks in homology groups (and fundamental groups) are regarded as Reeb spaces of simple fold maps.

The author has obtained (families of) Reeb spaces of explicit fold maps mainly for fold maps such that the restrictions to the singular sets are embeddings.
The present paper presents Reeb spaces of simple fold maps whose restrictions to the singular sets may not be embeddings in the class of $2$-dimensional branched surfaces. Such a class of Reeb spaces of fold maps has not been concentrated on so much.  For example, \cite{saeki3} shows that if a $3$-dimensional closed and orientable manifold has a simple fold map into a surface, then it has a simple fold map into the plane such that the restriction to the singular set is an embedding. This seems to imply explicitly that in considerable cases "simple fold maps" can be replaced by "simple fold maps such that the restrictions to the singular sets are embeddings".
\end{Rem}
\begin{Rem}
In Definition \ref{def:2}, in various situations, we also consider a kind of {\it orientations} for {\it branched points} or {\it orientations} around points corresponding to $0 \in K$ of the fibers. In the present paper, we do not consider such a notion.
We can consider Main Problem \ref{mprob:1} respecting this notion and this makes easier to find Reeb spaces which can not be Reeb spaces of simple fold maps. See \cite{saeki4} and \cite{yonebayashi} for example.
\end{Rem}
\section{Acknowledgment and data availability.}
The author is a member of and supported by JSPS KAKENHI Grant Number JP17H06128 "Innovative research of geometric topology and singularities of differentiable mappings" (Principal Investigator: Osamu Saeki). 

The author would like to thank Takahiro Yamamoto for discussions on systematic construction of explicit smooth maps via fundamental operations respecting structures of manifolds and maps, closely related to \cite{yamamoto_t} for example. This has lead the author to the present study. These discussions have motivated the author to systematically construct explicit smooth maps whose codimensions are negative and obtain their Reeb spaces and algebraic topological or differential topological information of them in \cite{kitazawa2, kitazawa3, kitazawa4, kitazawa6, kitazawa8, kitazawa9, kitazawa11} for example as in the introduction. These studies by the author have also played important roles in the birth of the present study. The author would also like to thank him for recent discussions on the present paper: the author could know that he had considered extending his theory to multibranched surfaces, systematically studied by Makoto Ozawa in \cite{ozawa} with \cite{matsuzakiozawa, munozozawa}. We can regard our study as a related pioneering study contributing to Yamamoto's theory on global (differential) topological properties on differentiable maps between surfaces.

The present study is also related to a joint research project at Institute of Mathematics for Industry, Kyushu University (20200027), ''Geometric and constructive studies of higher dimensional manifolds and applications to higher dimensional data'', principal investigator of which is the author and the author would like to thank people supporting the research project.
 
This is a project to apply geometric theory on higher dimensional differentiable manifolds developed through studies of the author to analysis of higher dimensional data analysis and visualizations for example. 

Various projects applying mathematics such as the singularity theory of differentiable maps and differential topology to machine-learnings and related problems such as multi-optimization problems, evolutionary computations and visualizations, have been established from Institute of Mathematics for Industry, Kyushu University, for which the author works, and around the world.
\cite{hamadahayanoichikikabatateramoto, ichikihamada, sakuraisaekicarrwuyamamotoduketakahashi} present related studies. The author has been also interested in this field. 
Discussions with Naoki Hamada, an expert of multi-optimization problems together with related topics, on applications of Reeb spaces to understanding topological structures
 of multi-functions in the real world and visualizations, have been very interesting. These discussions have also motivated the author to investigate global topologies of Reeb spaces of smooth maps of good classes further and lead the author to the present study.

The author would like to thank Makoto Ozawa for very interesting talks on embeddings of multibranched surfaces
into $3$-dimensional manifolds and interesting comments on questions by the author. 
He kindly introduced \cite{matsuzakiozawa} and Theorem 3.5 there as an important inequality on the genera of $3$-dimensional closed, connected and orientable manifolds into which we can embed multibranched surfaces. He also introduced the fact that we can consider the notion of an {\it immersion} of a multibranched surface and that this inequality can be extended for these immersions.
This has motivated the author to relate the present study to the topological theory multibranched surfaces immersed or embedded in $3$-dimensional manifolds as in Theorem \ref{thm:2} (\ref{thm:2.2}) and Theorem \ref{thm:3} and give answers to Main Problem \ref{mprob:3}.  

The author declares that data supporting the present study are all in the present paper. 


\begin{thebibliography}{30}
\bibitem{bryant} J. L. Bryant, \textsl{Piecewise linear topology}, https://www.maths.ed.ac.uk/~v1ranick/papers/pltop.pdf.
\bibitem{chess} D. Chess, \textsl{A Poincar\'e-Hopf type theorem for the de Rham invariant}, Bull. Amer. Math. Soc. 3 (1980), 1031--1035.
\bibitem{chillingworth} D. R. Chillingworth, \textsl{Winding number on surfaces. I}, Math. Ann., 196 (1972), 218--249.
\bibitem{eliashberg} Y. Eliashberg, \textsl{On singularities of folding type}, Math. USSR Izv. 4 (1970). 1119--1134.
\bibitem{eliashberg2} Y. Eliashberg, \textsl{Surgery of singularities of smooth mappings}, Math. USSR Izv. 6 (1972). 1302--1326.
\bibitem{golubitskyguillemin} M. Golubitsky and V. Guillemin, \textsl{Stable mappings and their singularities}, Graduate Texts in Mathematics (14), Springer-Verlag (1974).
\bibitem{hempel} J. Hempel, \textsl{3- Manifolds}, AMS Chelsea Publishing, 2004. 
\bibitem{hamadahayanoichikikabatateramoto} N. Hamada, K. Hayano, Y. Kabata, S. Ichiki and H. Teramoto, \textsl{Topology of Pareto sets of strongly convex problems}, SIAM Journal on Optimization, 30 , no. 3, pp. 2659--2686 arXiv:1904.03615. 
\bibitem{hatcher} A. Hatcher, \textsl{Algebraic Topology}, A modern, geometrically flavored introduction to algebraic topology, Cambridge: Cambridge University Press (2002).  
\bibitem{hudson} J. F. P. Hudson, \textsl{Piece Linear Topology}, W. A. Benjamin, New York, 1969.
\bibitem{ichikihamada} N. Hamada and S. Ichiki, \textsl{Simpliciality of strongly convex problems}, to appear in Journal of the Mathematics Society of Japan, arXiv:1912:09328.
\bibitem{ikeda} H. Ikeda, \textsl{Acyclic fake surfaces}, Topology 10. (1971), .9--36.
\bibitem{ishikawakoda} M. Ishikawa and Y. Koda, \textsl{Stable maps and branched shadows of $3$-manifolds}, Mathematische Annalen 367 (2017), no. 3, 1819--1863, arXiv:1403.0596.
\bibitem{kauffman} L. Kauffman, \textsl{Planar surface immersions}, Illinois J. Math., 23 (1979), 648--665.
\bibitem{kitazawa0.1} N. Kitazawa, \textsl{On round fold maps} (in Japanese), RIMS Kokyuroku Bessatsu B38 (2013), 45--59.
\bibitem{kitazawa0.2} N. Kitazawa, \textsl{On manifolds admitting fold maps with singular value sets of concentric spheres}, Doctoral Dissertation, Tokyo Institute of Technology (2014).
\bibitem{kitazawa0.3} N. Kitazawa, \textsl{Fold maps with singular value sets of concentric spheres}, Hokkaido Mathematical Journal Vol.43, No.3 (2014), 327--359.
\bibitem{kitazawa0.4} N. Kitazawa, \textsl{Constructions of round fold maps on smooth bundles}, Tokyo J. of Math. Volume 37, Number 2, 385--403, arxiv:1305.1708.
\bibitem{kitazawa0.5} N. Kitazawa, \textsl{On Reeb graphs induced from smooth functions on $3$-dimensional closed orientable manifolds with finitely many singular values}, Topological Methods in Nonlinear Analysis Vol. 59 No. 2B, 897--912, https://arxiv.org/abs/1902.08841, arxiv:1902:8841.
\bibitem{kitazawa0.6} N. Kitazawa, N. Kitazawa, \textsl{On Reeb graphs induced from smooth functions on closed or open manifolds}, a positive referee report for publication was announced to be sent and this will be published in Methods of Functional Analysis and Topology, arxiv:1908.04340.
\bibitem{kitazawa1} N. Kitazawa, \textsl{Round fold maps and the topologies and the differentiable structures of manifolds admitting explicit ones}, submitted to a refereed journal, arXiv:1304.0618 (the title has changed).



\bibitem{kitazawa2} N. Kitazawa, \textsl{Constructing fold maps by surgery operations and homological information of their Reeb spaces}, submitted to a refereed journal, arxiv:1508.05630 (the title has been changed).%

\bibitem{kitazawa3} N. Kitazawa, \textsl{Notes on fold maps obtained by surgery operations and algebraic information of their Reeb spaces}, arxiv:1811.04080.

\bibitem{kitazawa4} N. Kitazawa, \textsl{Explicit remarks on the torsion subgroups of homology groups of Reeb spaces of explicit fold maps}, submitted to a refereed journal, arxiv:1906.00943.
\bibitem{kitazawa5} N. Kitazawa, \textsl{Maps on manifolds onto graphs locally regarded as a quotient map onto a Reeb space and a new construction problem}, submitted to a refereed journal, arxiv:1909.10315.
\bibitem{kitazawa6} N. Kitazawa, \textsl{New observations on cohomology rings of Reeb spaces of explicit fold maps and manifolds admitting these maps}, arxiv:1911.09164.
\bibitem{kitazawa7} N. Kitazawa, \textsl{Notes on explicit smooth maps on 7-dimensional manifolds into the 4-dimensional Euclidean space}, submitted to a refereed journal, arxiv:1911.11274.
\bibitem{kitazawa8} N. Kitazawa, \textsl{Surgery operations to fold maps to construct fold maps whose singular value sets may have crossings}, arxiv:2003.04147.
\bibitem{kitazawa9} N. Kitazawa, \textsl{Surgery operations to fold maps to increase connected components of singular sets by two}, arxiv:2004.03583.
\bibitem{kitazawa10} N. Kitazawa, \textsl{Explicit fold maps on 7-dimensional closed and simply-connected manifolds of new classes}, arxiv:2005.05281.
\bibitem{kitazawa11} N. Kitazawa \textsl{Global topologies of Reeb spaces of stable fold maps with non-trivial top homology groups}, arXiv:2105.12934.%
\bibitem{kobayashisaeki} M. Kobayashi and O. Saeki, \textsl{Simplifying stable mappings into the plane from a global viewpoint}, Trans. Amer. Math. Soc. 348 (1996), 2607--2636. 
%
\bibitem{masumotosaeki} Y. Masumoto and O. Saeki, \textsl{A smooth function on a manifold with given Reeb graph}, Kyushu J. Math. 65 (2011), 75--84.
\bibitem{matsuzakiozawa} S. Matsuzaki and M. Ozawa, \textsl{Genera and minors of multibranched surfaces}, Topology Appl. 230 (2017), 621--638.
\bibitem{milnor} J. Milnor, \textsl{Morse Theory}, Annals of Mathematics Studies AM-51, Princeton University Press; 1st Edition (1963.5.1).
\bibitem{milnor2} J. Milnor, \textsl{Lectures on the h-cobordism theorem}, Math. Notes, Princeton Univ. Press, Princeton, N.J. 1965.
\bibitem{michalak} L. P. Michalak, \textsl{Realization of a graph as the Reeb graph of a Morse function on a manifold}. to appear in Topol. Methods Nonlinear Anal., Advance publication (2018), 14pp, arxiv:1805.06727.
\bibitem{moise} E E. Moise, \textsl{Affine Structures in $3$-Manifold{\rm :} V. The Triangulation Theorem and Hauptvermutung}, Ann. of Math., Second Series, Vol. 56, No. 1 (1952), 96--114.
\bibitem{munozozawa} M. E. Munoz and M. Ozawa, \textsl{The maximum and minimum genus of a multibranched surface}, Topology and its Appl. (2020) 107502, arXiv:2005,06765.
\bibitem{naoe} H. Naoe, \textsl{Shadows of $4$-manifolds with complexity zero and polyhedral collapsing}, Proc. Amer. Math. Soc. 145 (2017), 4561--4572.
\bibitem{reeb} G. Reeb, \textsl{Sur les points singuliers d\`{u}ne forme de Pfaff completement integrable ou d’une fonction numerique}, -C. R. A. S. Paris 222 (1946), 847--849. 
\bibitem{ozawa} M. Ozawa, \textsl{Multibranched surfaces in $3$-manifolds}, J. Math. Sci. 255 (2021), 193--208, arXiv:2005,07409.
\bibitem{pignoni} R. Pignoni, \textsl{Projections of surfaces with a connected fold curve}, Topology Appl. 49 (1) (1993), 55--74.
\bibitem{quine} J. R. Quine, \textsl{A global theorem for singularities of maps between oriented $2$-manifolds}, trans. Amer. Math. Soc. 236 (1978), 307--314.
\bibitem{saeki} O. Saeki, \textsl{Notes on the topology of folds}, J. Math. Soc. Japan Volume 44, Number 3 (1992), 551--566.
\bibitem{saeki2} O. Saeki, \textsl{Topology of special generic maps of manifolds into Euclidean spaces}, Topology Appl. 49 (1993), 265--293.
\bibitem{saeki3} O. Saeki, \textsl{Simple stable maps of $3$-manifolds into surfaces}, Topology 35, No.3 (1996), 671--698.
\bibitem{saeki4} O. Saeki, \textsl{Simple stable maps of $3$-manifolds into surfaces II}, J. Fac. Sci. Univ. Tokyo Sect. IA, Math. 40 (1993), 73--124.
\bibitem{saeki5} O. Saeki, \textsl{Reeb spaces of smooth functions on manifolds}, Intermational Mathematics Research Notices, maa301, https://doi.org/10.1093/imrn/maa301, arxiv:2006.01689.
\bibitem{saekisuzuoka} O. Saeki and K. Suzuoka, \textsl{Generic smooth maps with sphere fibers} J. Math. Soc. Japan Volume 57, Number 3 (2005), 881--902.
\bibitem{sakuma} K. Sakuma, \textsl{On the topology of simple fold maps}, Tokyo J. of Math. Volume 17, Number 1 (1994), 21--32.
\bibitem{sakuraisaekicarrwuyamamotoduketakahashi} D. Sakurai, O. Saeki, H. Carr, H Wu, T. Yamamoto, D. Duke and S. Takahashi, \textsl{Interactive Visualization for Singular Fibers of Functions $f:{\mathbb{R}}^3 \rightarrow {\mathbb{R}}^2$}, IEEE Transactions on Visualization and Computer Graphics ( Volume: 22, Issue: 1, Jan. 31 2016), 945--954.
\bibitem{sharko} V. Sharko, \textsl{About Kronrod-Reeb graph of a function on a manifold}, Methods of Functional Analysis and Topology 12 (2006), 389--396.
\bibitem{shiota} M. Shiota, \textsl{Thom's conjecture on triangulations of maps}, Topology 39 (2000), 383--399.
\bibitem{turaev} V. G. Turaev, \textsl{Shadow links and face models of statistical mechanics}, J. Differential Geom. 36 (1992), 35--74.
\bibitem{whitney} H. Whitney \textsl{On regular closed curves in the plane}, Comp. Math. 4 (1937), 276--284.
\bibitem{whitney2} H. Whitney, \textsl{On regular families of curves}, Bull. Amer. Math. Soc. 47 (1941), 145--147.
\bibitem{yamamoto_m} M. Yamamoto, \textsl{Immersions of surfaces with boundary into the plane}, Pacific. J. Math. Vol. 212 No.2 (2003), 371--376.
\bibitem{yamamoto_t} T. Yamamoto, \textsl{Number of Singularities of stable maps on surfaces},  Pacific. J. Math. Vol. 280 No.2 (2016), 489--510.
\bibitem{yonebayashi} Y. Yonebayashi, \textsl{Note on simple stable maps of $3$-manifolds into surfaces}, Osaka. J. Math. 36 (3) (1999), 685--709.
\end{thebibliography}
\end{document}